\documentclass{amsart}

\usepackage{amsmath}
\usepackage{verbatim}
\usepackage{upref}
\usepackage{amsfonts,amssymb}
\usepackage{youngtab}

\newtheorem{theorem}{Theorem}[section]

\newtheorem{proposition}[theorem]{Proposition}
\newtheorem{corollary}[theorem]{Corollary}

\newtheorem{main lemma}[theorem]{Main Lemma}

\theoremstyle{definition}
\newtheorem{definition}[theorem]{Definition}

\newcommand{\N}{\mathbb{N}}

\newcommand{\Z}{\mathbb{Z}}

\newcommand{\zn}{\Z_n}

\newcommand{\gk}{\text{\rm GKdim}}

\newcommand{\lie}{\mathcal{L}_k}
\newcommand{\K}{\mathbb{K}}

\begin{document}
\title[On the growth of graded polynomial identities of $sl_n$]{On the growth of graded polynomial identities of $sl_n$}

\author[Centrone]{Lucio Centrone}\address{Departamento de Matem\'atica
\\IMECC-UNICAMP \\
Rua S\'ergio Buarque de Holanda 651\\
13083-859 Campinas-SP, Brazil}
\email{centrone@ime.unicamp.br}

\author[Souza]{Manuela da Silva Souza}\address{Departamento de Matem\'atica
\\IME-USP \\
Rua do Mat\~ao, 1010, Cidade Universit\'aria\\
05508-090 S\~ao Paulo-SP, Brazil} \email{manuela@ime.usp.br}

\thanks{The first author was supported by grant from FAPESP (Proc. 2013/06752-4) of
Brazil}
\thanks{The second author was partially supported by grant from FAPESP (Proc. 2013/04590-7) of
Brazil.}


\begin{abstract}
Let $\K$ be a field of characteristic 0 and $L$ be a $G$-graded Lie PI-algebra, where $G$ is a finite group. We define the graded Gelfand-Kirillov
dimension of $L$. Then we measure the growth of the $\Z_n$-graded
polynomial identities of the Lie algebra of $n \times n$ traceless
matrices $sl_n(\K)$ giving an exact value of its $\Z_n$-graded Gelfand-Kirillov
dimension.
\end{abstract}

\maketitle

\section{introduction}All fields we refer to are to be considered of characteristic 0.
Let $A$ be an algebra with non-trivial polynomial identity (or
simply PI-algebra) and denote by $T(A)$ the $T$-ideal of its
polynomial identities. In general the description of a $T$-ideal is
a hard problem. In order to overcome this difficulty one introduces
some functions measuring the growth of $T(A)$ in some sense. For
every integer $k \geq 1$ one defines the Gelfand-Kirillov (GK)
dimension of $A$ in $k$ variables as the Gelfand-Kirillov dimension of the relatively
free algebra of rank $k$ with respect to the ideal of polynomial
identities of $A$. The Gelfand-Kirillov dimension of arbitrary
finitely generated algebra is a measure of the rate of growth in
terms of any finite generating set. For any associative PI-algebra
the GK dimension of $A$ in $k$ variables is always an integer (see
\cite{belov} and \cite{ber3}). For a more detailed background about
Gelfand-Kirillov dimension see the book of Krause and Lenagan \cite{krl1}.

After the powerful theory developed by Kemer in the 80's in order to solve the
\textit{Specht problem}, i.e., the existence of a finite generating set for any
$T$-ideal over a field of characteristic zero, other kinds of identities
have become object of much interest in the theory of polynomial
identities. For example if $A$ is a $G$-graded algebra, where $G$ is a group, one may be interested in the study of $G$-graded polynomial identities of $A$. In \cite{alb1} Aljadeff and Kanel-Belov proved the
analog of the Specht problem in the graded case, for $G$-graded PI-algebra when $G$ is a finite group. We have to cite the work by Sviridova \cite{svi1} who solved the Specht problem for associative graded algebras graded by a finite abelian group. 

When the group $G$
is finite, Centrone in \cite{cen1} defined the
$G$-graded GK dimension in $k$ graded variables for a $G$-graded PI-algebra. In particular if $A$ is a $G$-graded PI algebra, one may consider the relatively-free
$G$-graded algebra of $A$ in $k$ variables $\mathcal{F}_k^G(A)$ and define the $G$-graded
Gelfand-Kirillov dimension of $A$ as the GK dimension of $\mathcal{F}_k^G(A)$. In
\cite{cen2} Centrone computed the graded Gelfand-Kirillov dimension
of the \textit{verbally prime} algebras $M_n(\mathbb{K})$, $M_n(E)$ and
$M_{a,b}(E)$, where $\K$ is a field and $E$ is the infinite dimensional Grassmann algebra, endowed with a ``Vasilovsky type''-grading (see \cite{vas1}). We recall that the verbally prime algebras are the building blocks of the theory of Kemer.

The graded polynomial
identities of Lie algebras have seldom been studied. Most of the
known results about the topic are related to the algebra of
$2\times2$ traceless matrices $sl_2$. Razmyslov found
a finite basis of the ordinary identities satisfied by $sl_2$ and
proved that the variety of Lie algebras generated by $sl_2$ is
Spechtian i.e., the identities of any subvariety have a finite basis
(\cite{razmyslov}). Up to graded isomorphism,
$sl_2$ can only be graded by $\{0\}$ , $\mathbb{Z}_2 $,
$\mathbb{Z}_2 \times \mathbb{Z}_2$ and $\mathbb{Z}$. The structure of the relatively-free algebra for
each non-trivial grading over a field of characteristic $0$ was
described by Repin in \cite{repin}. In \cite{koshlukov} Koshlukov
described the graded polynomial identities of $sl_2$ for the above gradings
when the base field is infinite and of characteristic $\neq 2$.
Recently Giambruno and Souza proved in \cite{giambruno/souza} that
the variety of graded Lie algebras generated by $sl_2$ is also
Spechtian.

In this paper we define the graded Gelfand-Kirillov dimension on a finite set of graded variables of any
$G$-graded Lie PI-algebra according to the associative case. We compute
explicitly the graded Gelfan-Kirillov dimension of $sl_2$ over a
field of characteristic $0$ for each of its non-trivial
gradings. Next we consider the Lie algebra $sl_n$ of $n\times n$ traceless matrices endowed with the grading of Vasilovsky, then we measure the growth of the graded polynomial identities
of the Lie algebra of $n \times n$ traceless matrices $sl_n$ giving
the exact value of their graded Gelfand-Kirillov dimension. We recall that the ordinary GK dimension in $k$ variables of the relatively free algebra of $sl_2$ has been already computed with two different methods in \cite{dkm1} and \cite{mak1}.

\section{Preliminaries}
Throughout the paper $\mathbb{K}$ will be a fixed field of
characteristic zero, $L$ a Lie algebra over $\mathbb{K}$ and $sl_n =
sl_n(\mathbb{K})$ the Lie algebra of $n \times n$ traceless matrices
over $\mathbb{K}$. Moreover we refer to a \textit{left-normed} product of elements $l_1,\ldots,l_k$ of a Lie algebra as the product $[[\cdots[[l_1,l_2],l_3]\cdots],l_k]$. 

\

Recall that the growth function $g_V(n)(A)$ of a finitely generated
algebra $A$ generated by the set $V=\{v_1, \ldots, v_r\}$ is the
dimension of the space spanned by words of length at most $n$. We say
that $g_V(n)(A)$ is the growth function with respect to the vector
space generated by $V$ over $\K$. The
Gelfand-Kirillov dimension of $A$ is the superior limit
$$\text{\rm GKdim$(A)=\limsup_{n\rightarrow+\infty}\frac{\ln
g_V(n)(A)}{\ln n}$}$$ when it exists. The Gelfand-Kirillov dimension
does not depend on the choice of the generators of algebra, then we shall use $g(n)(A)$ instead of $g_V(n)(A)$.

From now on $G$ will be an abelian group and $L$ be a $G$-graded Lie
algebra over $\mathbb{K}$. Recall that $L$ is a $G$-graded Lie
algebra if $L = \bigoplus_{g \in G} L_g$ is a direct sum of
subspaces such that $L_g L_h \subseteq L_{g + h},$ for all $g, h \in
G$.

The free $G$-graded Lie algebra $\mathcal{L}(X)$ is the $G$-graded
Lie algebra freely generated on the set $X = \bigcup_{g \in G} X_g$,
where for any $g \in G$ the sets $X_g = \{x_i^{g}; i \geq 1\}$ of
variables of homogeneous degree $g$ are infinite and disjoint. A
polynomial $f$ of $\mathcal{L}(X)$ is a $G$-graded polynomial
identity of $L$ if $f$ vanishes under all graded substitutions i.e.,
for any $g \in G$, we evaluate the variables $x_i^{g}$ into elements
of the homogeneous component $L_g$. We denote by $T_{G}(L)$ the
ideal of $\mathcal{L}(X)$ of $G$-graded polynomial identities of
$L$. It is easily checked that $T_{G}(L)$ is a $T_G$-ideal i.e., an
ideal invariant under all $G$-graded endomorphisms of
$\mathcal{L}(X)$. We say that $L$ is a graded PI-algebra if $T_{G}(L) \neq
0$.

Here we shall always consider gradings on $L$ such that the support
$\mbox{Supp}(L) = \{g \in G: L_g \neq 0\}$ is finite. Suppose for
instance that $\mbox{Supp}(L) = \{g_1, \ldots, g_s\}$. We denote by
\[\mathcal{L}_k^{G}(L) = \dfrac{\mathcal{L}\left(
x_1^{g_1},\ldots,x_k^{g_1},\ldots,x_1^{g_s},\ldots,x_k^{g_s}\right)}{\mathcal{L}\left(
x_1^{g_1},\ldots,x_k^{g_1},\ldots,x_1^{g_s},\ldots,x_k^{g_s}\right)\cap
T_{G}(L)}\] the relatively-free $G$-graded algebra of $L$ in $k$
variables. We define the $G$-graded Gelfand-Kirillov dimension of $L$ in $k$
variables similarly to the associative case (see \cite{cen1,cen2}), 
\[{\rm GKdim}^G_k(L):={\rm GKdim}(\mathcal{L}_k^G(L)).\]See the survey of Drensky \cite{dre2} and Centrone \cite{cen4} for more details on the Gelfand-Kirillov dimension of PI-algebras.


Suppose that $L$ is a $G$-graded PI-algebra. It is well known that
the relatively-free $G$-graded algebra $\mathcal{L}_k^{G}(L)$ in the
variables $\bar{x}_1^{g_1}, \ldots, \bar{x}_k^{g_1},
\bar{x}_1^{g_2}, \ldots, \bar{x}_k^{g_2}, \ldots, \bar{x}_1^{g_s},
\ldots, \bar{x}_k^{g_s}$, where $\bar{x}_i^g$ is the homomorphic
image of $x_i^g$ in $\mathcal{L}_k^{G}(L)$, is a
$\mathbb{Z}^{\infty}$-graded Lie algebra when considering the usual
multigrading. The \textit{Hilbert series} of $\mathcal{L}_k^{G}(L)$ in the
variables $t_1, t_2, \ldots, t_r$, $r = sk$ is the formal power
series
\[H(L, t_1, t_2, \ldots, t_s)= \sum_{m = (m_1, m_2, \ldots,
m_s)}\dim \mathcal{L}_k^{G}(L)^{(m_1, m_2, \ldots,
m_s)}t_1^{m_1}t_2^{m_2}\cdots t_s^{m_s},\] where
$\mathcal{L}_k^{G}(L)^{(m_1, m_2, \ldots, m_s)}$ is the homogeneous
component of degree $(m_1, m_2, \ldots, m_s)$.

The growth function of $\mathcal{L}_k^{G}(L)$ with respect to the
vector space generated by \[V = \{\bar{x}_1^{g_1}, \ldots,
\bar{x}_k^{g_1},\ldots, \bar{x}_1^{g_s}, \ldots, \bar{x}_k^{g_s}\}\]
is
\begin{equation*} \label{growth_function} g(n)(\mathcal{L}_k^G(sl_2)) =
\sum_{m \leq n} a_m,
\end{equation*}
where
\begin{equation}\label{a_m}
a_m = \displaystyle \sum_{m = m_1 + m_2 + \cdots + m_s}
\dim_\mathbb{K} \mathcal{L}_k^{G}(L)^{(m_1, m_2, \ldots, m_s)}.
\end{equation}

Let $T$ be a tableau of shape $\sigma$ filled in with natural
numbers $\{1,\ldots,k\}$ and let $d_i$ be the multiplicity of $i$ in
$T$. A tableau is said to be \emph{semistandard} if the entries
weakly increase along each row and strictly increase down each
column. We say that
$$S_{\sigma}(t_1,\ldots,t_k) = \sum_{\text{\rm $T_{\sigma}$
semistandard}}{t_1^{d_1}t_2^{d_2}\cdots t_k^{d_k}}$$ is the
\emph{Schur Function of $\sigma$ in the variables $t_1,\ldots,t_k.$}

The $G$-graded cocharacters of $L$ are strictly related with the
Hilbert series $H(L, t_1, t_2, \ldots, t_r)$. We have the following
(see \cite{ber2},\cite{dre1}).

\begin{proposition}\label{prop16} Let $L$ be a $G$-graded Lie algebra and $m_1$,$\ldots$, $m_s \geq 0$. Suppose
that \begin{equation*}\chi_{m_1, \ldots,m_s}(L) = \sum_{\left\langle
\sigma \right\rangle \vdash m}m_{\left\langle \sigma \right\rangle}
\chi_{\sigma_1}\otimes \cdots \otimes \chi_{\sigma_s}
\end{equation*}
is the $(m_1, \ldots,m_s)$-cocharacter of $L$ where $\left\langle
\sigma \right\rangle = (\sigma_1, \ldots, \sigma_s)$ is a
multipartition of $n$ i.e., $\sigma_1 \vdash m_1, \ldots, \sigma_s
\vdash m_s$ and $m = m_1 + \cdots + m_s$. Then
\begin{equation}\label{hilbertseries_schur}H(L,
t_1, t_2, \ldots, t_r) =\sum_{m = (m_1, m_2, \ldots, m_r)}
m_{\sigma_1,\ldots,\sigma_s}S_{\sigma_1}(T_1)\cdots
S_{\sigma_s}(T_s)
\end{equation}
where $S_{\sigma_1}(T_1)$, $\ldots$,
$S_{\sigma_s}(T_s)$ are Schur functions with shape $\sigma_1,
\ldots, \sigma_s$ in the variables $T_1 = \{t_1, t_2, \ldots,
t_k\}$, $\ldots$, $T_s = \{t_{sk - k + 1}, t_{sk - k + 2}, \ldots,
t_r \}$.
\end{proposition}

\section{Graded Gelfand-Kirilov dimension for $sl_2$}
In this section we calculate the $G$-graded Gelfand-Kirillov
dimension of $sl_2$, for any non-trivial group $G$.
\

In order to simplify the notation we write
$$h = \begin{pmatrix} 1 & 0 \\ 0 &-1 \end{pmatrix},\quad\quad e =
\begin{pmatrix} 0 & 1 \\ 0 &0 \end{pmatrix}, \quad\quad f =
\begin{pmatrix} 0 & 0 \\ 1 &0 \end{pmatrix},$$ for the standard
basis of $L = sl_2$. Up to equivalence, $sl_2$ has three non-trivial
$G$-gradings (see \cite{bahturin}):
\begin{enumerate}
\item
$G=\mathbb{Z}_2$: $L_0=\mathbb{K}h, \, L_1 = \mathbb{K}e \oplus
\mathbb{K}f;$
\item
$G=\mathbb{Z}_2\times \mathbb{Z}_2$: $L_{(0,0)} = 0,\, L_{(1,0)}=
\mathbb{K}h,\, L_{(0,1)} = \mathbb{K}(e + f),\, L_{(1,1)}=
\mathbb{K}(e - f);$
\item
$G=\mathbb{Z}$: $L_{-1}= \mathbb{K}e, \, L_{0} = \mathbb{K}h, \,
L_{1}= \mathbb{K}f, \, L_{i} = 0, \;  i \notin \{-1, 0, 1\}$.
\end{enumerate}

We shall write $sl_2^G$ for the algebra $sl_2$ endowed with the
correspondent $G$-grading (when necessary). We consider each
grading separately.

We recall that $\lfloor w \rfloor = \max\{u \in \mathbb{Z}\; : \; u
\leq w\}$ is the integer part of $w$ when $w$ is a real number.

\subsection{The $\mathbb{Z}_2$-grading}

We assume that $L = sl_2$ is $\mathbb{Z}_2$-graded
i.e., $L = L_0 \oplus L_1$ where $L_0 = \mathbb{K}h$ and $L_1 =
\mathbb{K}e \oplus \mathbb{K}f$. In order to compute the
$\mathbb{Z}_2$-graded Gelfand-Kirillov dimension of $sl_2$ we use
the Repin's description of the $\mathbb{Z}_2$-graded cocharacter
sequence of $sl_2$.

\begin{proposition} [\cite{repin}] \label{repin1} Let
$$\chi_{p, n-p}(sl_2^{\mathbb{Z}_2}) = \sum_{\sigma \vdash p,\, \tau \vdash n-p} m_{\sigma, \tau} \chi_{\sigma}\otimes\chi_{\tau}$$
be the $(p, n-p)$-th cocharacter of $sl_2$. Then, for every $\sigma
\vdash p$ and $\tau \vdash n-p$, $m_{\sigma, \tau} \leq 1$.
Moreover, $m_{\sigma, \tau} = 1$ if and only if $\sigma = (p)$,
$\tau = (q+r, q)$ and the following conditions hold:
\begin{enumerate}
\item $p \neq n$;
\item $r \neq n$;
\item $r \equiv 1$ or $p+q \equiv 1$(mod 2).
\end{enumerate}
\end{proposition}

Then we have the following result.

\begin{proposition} \label{propGK1} The growth function $g(n) = g(n)(\mathcal{L}_k^{\Z_2}(sl_2))$ of $\mathcal{L}_k^{\Z_2}(sl_2)$ is
a polynomial in $n$ of degree $3k-1$.
\end{proposition}
\begin{proof} By Proposition \ref{repin1},
$$\chi_{p, m-p}(sl_2^{\mathbb{Z}_2}) = \sum
\begin{tabular}{|l|l|l|l|}
  \hline
  \multicolumn{2}{|l|}{$p$}\\
  \cline{1-2}
\end{tabular}\otimes\begin{tabular}{|l|l|l|l|}
  \hline
  \multicolumn{4}{|l|}{$m-p-q$}\\
  \hline
  \multicolumn{2}{|l|}{$q$}\\
  \cline{1-2}
\end{tabular}$$
where $p \neq m$, $r = m - p - 2q \neq m$ and $r \equiv 1$ or $p+q
\equiv 1$(mod 2). If $m$ is even we have that
\begin{eqnarray*}\lefteqn{\chi_{p,
m-p}(sl_2^{\mathbb{Z}_2}) = \sum_{p_1 = 0}^{\frac{m-2}{2}} \sum_{q_1
= 0}^{\lfloor\frac{m-(2p_1 +1)}{2}\rfloor}\begin{tabular}{|l|l|l|l|}
  \hline
  \multicolumn{2}{|l|}{$2p_1 +1$}\\
  \cline{1-2}
\end{tabular}\otimes\begin{tabular}{|l|l|l|l|}
  \hline
  \multicolumn{4}{|l|}{$m-(2p_1 +1)-q_1$}\\
  \hline
  \multicolumn{2}{|l|}{$q_1$}\\
  \cline{1-2}
\end{tabular}}\\
& & + \sum_{p_2 = 0}^{\frac{m-2}{2}} \sum_{t_2 =
0}^{\lfloor\frac{m-2p_2-2}{4}\rfloor}\begin{tabular}{|l|l|l|l|}
  \hline
  \multicolumn{2}{|l|}{$2p_2$}\\
  \cline{1-2}
\end{tabular}\otimes\begin{tabular}{|l|l|l|l|}
  \hline
  \multicolumn{4}{|l|}{$m-2p_2-(2t_2 +1)$}\\
  \hline
  \multicolumn{2}{|l|}{$2t_2+1$}\\
  \cline{1-2}
\end{tabular}.
\end{eqnarray*}
By the definition of Hilbert series given
in (\ref{a_m}) and (\ref{hilbertseries_schur}) we get:
\begin{eqnarray*}\lefteqn{a_m = \sum_{p_1 = 0}^{\frac{m-2}{2}} \sum_{q_1 =
0}^{\lfloor\frac{m-(2p_1 +1)}{2}\rfloor}
s_T\left(\begin{tabular}{|l|l|l|l|}
  \hline
  \multicolumn{2}{|l|}{$2p_1 +1$}\\
  \cline{1-2}
\end{tabular}\right)s_U\left(\begin{tabular}{|l|l|l|l|}
  \hline
  \multicolumn{4}{|l|}{$m-(2p_1 +1)-q_1$}\\
  \hline
  \multicolumn{2}{|l|}{$q_1$}\\
  \cline{1-2}
\end{tabular}\right)}\\
& & + \sum_{p_2 = 0}^{\frac{m-2}{2}} \sum_{t_2 =
0}^{\lfloor\frac{m-2p_2-2}{4}\rfloor}s_T\left(\begin{tabular}{|l|l|l|l|}
  \hline
  \multicolumn{2}{|l|}{$2p_2$}\\
  \cline{1-2}
\end{tabular}\right)s_U\left(\begin{tabular}{|l|l|l|l|}
  \hline
  \multicolumn{4}{|l|}{$m-2p_2-(2t_2 +1)$}\\
  \hline
  \multicolumn{2}{|l|}{$2t_2+1$}\\
  \cline{1-2}
\end{tabular}\right),
\end{eqnarray*}
where $s_T\left(\begin{tabular}{|l|l|l|l|}
  \hline
  \multicolumn{2}{|l|}{$p$}\\
  \cline{1-2}
\end{tabular}\right)$, $s_U\left(\begin{tabular}{|l|l|l|l|}
  \hline
  \multicolumn{4}{|l|}{$m-p-q$}\\
  \hline
  \multicolumn{2}{|l|}{$q$}\\
  \cline{1-2}
\end{tabular}\right)$ is the number of semistandard tableaux with shape $\begin{tabular}{|l|l|l|l|}
  \hline
  \multicolumn{2}{|l|}{$p$}\\
  \cline{1-2}
\end{tabular}$, $\begin{tabular}{|l|l|l|l|}
  \hline
  \multicolumn{4}{|l|}{$m-p-q$}\\
  \hline
  \multicolumn{2}{|l|}{$q$}\\
  \cline{1-2}
\end{tabular}$ in the variables $T$ and $U$ respectively. For $m$ odd we observe that $a_m$ can be calculated as above. Since (see \cite{ker1,wey1})
\begin{eqnarray*}
s_T\left(\begin{tabular}{|l|l|l|l|}
  \hline
  \multicolumn{2}{|l|}{$p$}\\
  \cline{1-2}
\end{tabular}\right) &=& \binom{p+k-1}{k-1}\\
s_U\left(\begin{tabular}{|l|l|l|l|}
  \hline
  \multicolumn{4}{|l|}{$m-p-q$}\\
  \hline
  \multicolumn{2}{|l|}{$q$}\\
  \cline{1-2}
\end{tabular}\right) &=&
\frac{m-p-2q+1}{k-1}\binom{m-p-q+k-1}{k-2}\binom{q+k-2}{k-2}
\end{eqnarray*}
we can assume that for all $m$, $a_m$ is asymptotically equivalent
to
\begin{eqnarray*} \sum_{p_1 = 0}^{\frac{m-2}{2}} \sum_{q_1 =
0}^{\lfloor\frac{m-(2p_1 +1)}{2}\rfloor}
s_T\left(\begin{tabular}{|l|l|l|l|}
  \hline
  \multicolumn{2}{|l|}{$2p_1 +1$}\\
  \cline{1-2}
\end{tabular}\right)s_U\left(\begin{tabular}{|l|l|l|l|}
  \hline
  \multicolumn{4}{|l|}{$m-(2p_1 +1)-q_1$}\\
  \hline
  \multicolumn{2}{|l|}{$q_1$}\\
  \cline{1-2}
\end{tabular}\right)=\\
\sum_{p_1 = 0}^{\frac{m-2}{2}} \binom{2p_1+k}{k-1} \sum_{q_1 =
0}^{\lfloor\frac{m-(2p_1 +1)}{2}\rfloor}
\frac{m-2p_1-2q_1}{k-1}\binom{m-2p_1-2q_1-k-2}{k-2}\binom{q_1+k-2}{k-2}.
\end{eqnarray*}
Using standard combinatorial arguments, it is easy to see that
$a_m$ is a polynomial in $m$ of degree $3k-2$. Therefore $g(n)$ is a
polynomial in $n$ of degree $3k-1$ since $g(n) = \displaystyle
\sum_{m \leq n} a_m$.
\end{proof}

As an immediate consequence of the previous proposition we obtain
the $\Z_2$-graded Gelfand-Kirillov dimension of $sl_2$.

\begin{corollary} $\gk_k^{\mathbb{Z}_2}(sl_2)=3k-1$.
\end{corollary}

\subsection{The $\mathbb{Z}_2 \times \mathbb{Z}_2$-grading}

We consider the $\mathbb{Z}_2 \times
\mathbb{Z}_2$-grading of $L = sl_2$ where $L_{(0, 0)} =0$, $L_{(1,
0)} = \mathbb{K}h$, $L_{(0,1)} = \mathbb{K}(e+f)$ and $L_{(1, 1)} =
\mathbb{K}(e-f)$.

The next result by Repin \cite{repin} shows the $\Z_2\times\Z_2$-cocharacters decomposition of $sl_2$.

\begin{proposition} \label{repin2} Let
$$\chi_n(sl_2^{\mathbb{Z}_2 \times \mathbb{Z}_2}) = \sum_{ \substack{ \sigma \vdash p, \, \tau \vdash q \\ \pi \vdash r}} m_{\sigma, \tau, \pi} \chi_{\sigma}\otimes\chi_{\tau}\otimes\chi_{\pi}$$
be the $(p, q, r)$-th cocharacter of $sl_2$. Then, for every $\sigma
\vdash p$, $\tau \vdash q$, $\pi \vdash r$, $m_{\sigma, \tau, \pi}
\leq 1$. Moreover, $m_{\sigma, \tau, \pi} = 1$ if and only if
$\sigma = (p), \tau = (q), \pi = (r)$ and the following conditions
hold:
\begin{enumerate}
\item $p \neq n$, $q \neq n$, $r \neq n$;
\item $p+q \equiv 1$ or $q+r \equiv 1$(mod 2).
\end{enumerate}
\end{proposition}

We have the following result.

\begin{proposition}\label{propGK2} The growth function $g(n) = g(n)(\mathcal{L}_k^{\Z_2 \times \Z_2}(sl_2))$ of $\mathcal{L}_k^{\Z_2 \times \Z_2}(sl_2)$ is
a polynomial in $n$ of degree $3k+1$.
\end{proposition}
\begin{proof} By Proposition \ref{repin2},
$$\chi_{p, q, r}(sl_2^{\mathbb{Z}_2 \times \mathbb{Z}_2}) = \sum
\begin{tabular}{|l|l|l|l|}
  \hline
  \multicolumn{2}{|l|}{$p$}\\
  \cline{1-2}
\end{tabular}\otimes\begin{tabular}{|l|l|l|l|}
  \hline
  \multicolumn{2}{|l|}{$q$}\\
  \cline{1-2}
\end{tabular} \otimes \begin{tabular}{|l|l|l|l|}
  \hline
  \multicolumn{2}{|l|}{$r$}\\
  \cline{1-2}
\end{tabular}$$
where $m = p + q + r$, $p \neq m$, $q \neq m$, $r \neq m$ and $p+q
\equiv 1$ or $q+r \equiv 1$(mod 2).

If $m$ is even then
\begin{eqnarray*}\lefteqn{\chi_{p,q, r}(sl_2^{\mathbb{Z}_2 \times \mathbb{Z}_2}) = \sum_{s_1 = 0}^{\lfloor\frac{m-1}{2}\rfloor} \sum_{t_1
= 0}^{\lfloor\frac{m-2s_1-1}{2}\rfloor} \sum_{u_1 =
0}^{\frac{m-2s_1-2t_1-2}{2}} \begin{tabular}{|l|l|l|l|}
  \hline
  \multicolumn{2}{|l|}{$2s_1$}\\
  \cline{1-2}
\end{tabular}\otimes\begin{tabular}{|l|l|l|l|}
  \hline
  \multicolumn{2}{|l|}{$2t_1 + 1$}\\
  \cline{1-2}
\end{tabular} \otimes \begin{tabular}{|l|l|l|l|}
  \hline
  \multicolumn{2}{|l|}{$2u_1+1$}\\
  \cline{1-2}
\end{tabular}}\\
& & + \sum_{s_2 = 0}^{\frac{m-2}{2}} \sum_{t_2 =
0}^{\frac{m-2s_2-2}{2}} \sum_{u_2 = 0}^{\frac{m-2s_2-2t_2-2}{2}}
\begin{tabular}{|l|l|l|l|}
  \hline
  \multicolumn{2}{|l|}{$2s_2 + 1$}\\
  \cline{1-2}
\end{tabular}\otimes\begin{tabular}{|l|l|l|l|}
  \hline
  \multicolumn{2}{|l|}{$2t_2 + 1$}\\
  \cline{1-2}
\end{tabular} \otimes \begin{tabular}{|l|l|l|l|}
  \hline
  \multicolumn{2}{|l|}{$2u_2$}\\
  \cline{1-2}
\end{tabular}\\
& & + \sum_{s_3 = 0}^{\frac{m-2}{2}} \sum_{t_3 =
0}^{\lfloor\frac{m-2s_3-1}{2}\rfloor} \sum_{u_3 =
0}^{\frac{m-2s_3-2t_3-2}{2}}
\begin{tabular}{|l|l|l|l|}
  \hline
  \multicolumn{2}{|l|}{$2s_3 +1$}\\
  \cline{1-2}
\end{tabular}\otimes\begin{tabular}{|l|l|l|l|}
  \hline
  \multicolumn{2}{|l|}{$2t_3$}\\
  \cline{1-2}
\end{tabular} \otimes \begin{tabular}{|l|l|l|l|}
  \hline
  \multicolumn{2}{|l|}{$2u_3+1$}\\
  \cline{1-2}
\end{tabular}.
\end{eqnarray*}
Using arguments similar to those of Proposition \ref{propGK1} we have that
\begin{eqnarray*} a_m &\approx&
\sum_{s_1 = 0}^{\lfloor\frac{m-1}{2}\rfloor} \sum_{t_1 =
0}^{\lfloor\frac{m-2s_1-1}{2}\rfloor} \sum_{u_1 =
0}^{\frac{m-2s_1-2t_1-2}{2}} s_T\left(\begin{tabular}{|l|l|l|l|}
  \hline
  \multicolumn{2}{|l|}{$2s_1$}\\
  \cline{1-2}
\end{tabular}\right)s_U\left(\begin{tabular}{|l|l|l|l|}
  \hline
  \multicolumn{2}{|l|}{$2t_1 + 1$}\\
  \cline{1-2}
\end{tabular}\right)s_V\left(\begin{tabular}{|l|l|l|l|}
  \hline
  \multicolumn{2}{|l|}{$2u_1+1$}\\
  \cline{1-2}
\end{tabular}\right)\\
&=& \sum_{s_1 = 0}^{\lfloor\frac{m-1}{2}\rfloor} \sum_{t_1 =
0}^{\lfloor\frac{m-2s_1-1}{2}\rfloor} \sum_{u_1 =
0}^{\frac{m-2s_1-2t_1-2}{2}} \binom{2s_1 + k-1}{k-1}\binom{2t_1 +
k}{k-1}\binom{2u_1 + k}{k-1}
\end{eqnarray*}
for all $m$. It is easy to see that $a_m$ is a polynomial in $m$ of
degree $3k$. Therefore $g(n)$ is a polynomial in $n$ of degree $3k
+1$.
\end{proof}

\begin{corollary} $\gk_k^{\mathbb{Z}_2 \times \mathbb{Z}_2}(sl_2)=3k+1$.
\end{corollary}

\subsection{The $\mathbb{Z}$-grading}

We consider now the $\mathbb{Z}$-grading of
$L=sl_2$, where $L_{-1} = \mathbb{K}e$, $L_0 = \mathbb{K}h$, $L_1 =
\mathbb{K}f$ and $L_i = 0$ for all $i \notin \{-1, 0, 1\}$.

We have the following result by Repin \cite{repin}.

\begin{proposition} \label{repin3} Let
$$\chi_n(sl_2^{\mathbb{Z}}) = \sum_{ \substack{ \sigma \vdash p, \, \tau \vdash q \\ \pi \vdash r}} m_{\sigma, \tau, \pi} \chi_{\sigma} \otimes \chi_{\tau}\otimes \chi_{\pi}$$
be the $(p, q, r)$-th cocharacter of $sl_2$.  Then, for every
$\sigma \vdash p$, $\tau \vdash q$, $\pi \vdash r$, $m_{\sigma,
\tau, \pi} \leq 1$. Moreover, $m_{\sigma, \tau, \pi} = 1$ if and
only if $\sigma = (p)$, $\tau = (q)$, $\pi = (r)$ and the following
conditions hold:
\begin{enumerate}
\item $p \neq n$, $q \neq n$, $r \neq n$;
\item $\left|p - r\right| \leq 1$.
\end{enumerate}
\end{proposition}

As in the previous two cases, we get the asympthotics for the growth function of $\mathcal{L}_k^{\Z}(sl_2)$.

\begin{proposition} The growth function $g(n) = g(n)(\mathcal{L}_k^{\Z}(sl_2))$ of $\mathcal{L}_k^{\Z}(sl_2)$ is
a polynomial in $n$ of degree $3k-1$.
\end{proposition}
\begin{proof} By Proposition \ref{repin3},
$$\chi_{p, q, r}(sl_2^{\mathbb{Z}}) = \sum
\begin{tabular}{|l|l|l|l|}
  \hline
  \multicolumn{2}{|l|}{$p$}\\
  \cline{1-2}
\end{tabular}\otimes\begin{tabular}{|l|l|l|l|}
  \hline
  \multicolumn{2}{|l|}{$q$}\\
  \cline{1-2}
\end{tabular} \otimes \begin{tabular}{|l|l|l|l|}
  \hline
  \multicolumn{2}{|l|}{$r$}\\
  \cline{1-2}
\end{tabular}$$
where $m = p + q + r$, $p \neq m$, $q \neq m$, $r \neq m$ and
$\left|p - r\right| \leq 1$.

For $m$ even we have that
\begin{eqnarray*}\lefteqn{\chi_{p,q, r}(sl_2^{\mathbb{Z}}) = \sum_{s_1 = 0}^{\lfloor\frac{m-1}{2}\rfloor}
\begin{tabular}{|l|l|l|l|}
  \hline
  \multicolumn{2}{|l|}{$\frac{m-2s_1}{2}$}\\
  \cline{1-2}
\end{tabular}\otimes\begin{tabular}{|l|l|l|l|}
  \hline
  \multicolumn{2}{|l|}{$2s_1$}\\
  \cline{1-2}
\end{tabular} \otimes \begin{tabular}{|l|l|l|l|}
  \hline
  \multicolumn{2}{|l|}{$\frac{m-2s_1}{2}$}\\
  \cline{1-2}
\end{tabular}}\\
& & + \sum_{s_2 = 0}^{\frac{m-2}{2}}
\begin{tabular}{|l|l|l|l|}
  \hline
  \multicolumn{2}{|l|}{$\lfloor \frac{m-2s_2-1}{2}\rfloor + 1$}\\
  \cline{1-2}
\end{tabular}\otimes\begin{tabular}{|l|l|l|l|}
  \hline
  \multicolumn{2}{|l|}{$2s_2+1$}\\
  \cline{1-2}
\end{tabular} \otimes \begin{tabular}{|l|l|l|l|}
  \hline
  \multicolumn{2}{|l|}{$\lfloor \frac{m-2s_2-1}{2}\rfloor$}\\
  \cline{1-2}
\end{tabular}\\
& & + \sum_{s_2 = 0}^{\frac{m-2}{2}}
\begin{tabular}{|l|l|l|l|}
  \hline
  \multicolumn{2}{|l|}{$\lfloor \frac{m-2s_2-1}{2}\rfloor$}\\
  \cline{1-2}
\end{tabular}\otimes\begin{tabular}{|l|l|l|l|}
  \hline
  \multicolumn{2}{|l|}{$2s_2+1$}\\
  \cline{1-2}
\end{tabular} \otimes \begin{tabular}{|l|l|l|l|}
  \hline
  \multicolumn{2}{|l|}{$\lfloor \frac{m-2s_2-1}{2}\rfloor + 1$}\\
  \cline{1-2}
\end{tabular}.
\end{eqnarray*}
Using arguments similar to those of Proposition \ref{propGK1} we obtain
\begin{eqnarray*} a_m &\approx&
\sum_{s_1 = 0}^{\frac{m-1}{2}} s_T\left(\begin{tabular}{|l|l|l|l|}
  \hline
  \multicolumn{2}{|l|}{$\frac{m-2s_1}{2}$}\\
  \cline{1-2}
\end{tabular}\right)s_U\left(\begin{tabular}{|l|l|l|l|}
  \hline
  \multicolumn{2}{|l|}{$2s_1$}\\
  \cline{1-2}
\end{tabular}\right)s_V\left(\begin{tabular}{|l|l|l|l|}
  \hline
  \multicolumn{2}{|l|}{$\frac{m-2s_1}{2}$}\\
  \cline{1-2}
\end{tabular}\right)\\
&=& \sum_{s_1 = 0}^{\frac{m-1}{2}} {\binom{\frac{m-2s_1}{2} +
k-1}{k-1}}^2 \binom{2s_1 + k-1}{k-1}.
\end{eqnarray*}
and consequently $g(n)$ is a polynomial in $n$ of degree $3k-1$.
\end{proof}

\begin{corollary} $\gk_k^{\mathbb{Z}}(sl_2)=3k-1$.
\end{corollary}

\section{The $\Z_n$-graded GK dimension of $sl_n$}
In this section we compute the exact value of the $\Z_n$-graded GK dimension of $sl_n$.

Let us consider the relatively-free
algebra of $sl_n$ graded by $\Z_n$ with the grading of Vasilovsky (see \cite{vas1} and \cite{div1}). Let $(sl_n)_i = \mbox{span}_\mathbb{K}\{e_{pq} :
|q-p| = i \},$ for all $i \in \{0, \ldots, n-1\}$, where the $e_{pq}$'s are
the usual matrix units and $|\cdot|$ is the reduction modulo $n$.

\

We shall construct $\lie^{\zn}(sl_n)$ as a generic matrix algebra. For this purpose we consider $X=\{x_{ij}^{(r)}| i,j=1,\ldots,n-1,r=1,\ldots,k\}$ and for every $r=1,\ldots,k$, let us consider the generic $n \times n$ matrices with entries from the algebra of the commutative polynomials $\K[X]$:
\begin{eqnarray*}
A_0^{r}
&=&\sum_{i=1}^{n-1}x_{ii}^{(r)}e_{ii}-\left(\sum_{i=1}^{n-1}x_{ii}^{(r)}\right)e_{nn},\\
A_i^{r} &=& \sum_{|q-p|=i}x_{pq}^{(r)}e_{pq}.
\end{eqnarray*}

We denote by $L_k(A)$ and $\K_k(A)$ respectively the Lie algebra and the associative unitary algebra generated by the $A_i^{r}$'s. let us observe that $L_k(A)$ is contained in $\K_k(A)$. We have the following results (see \cite{razmyslov}).

\begin{theorem} \label{teo1}
For every $n$, $k \in \N$, $k\geq 2$, we have $\lie^{\zn}(sl_n)\cong
L_k(A)$.
\end{theorem}

In \cite{bal1} the author proved the existence of a graded field of quotients $Q$ for associative unitary $G$-prime algebras provided $G$ to be either an abelian or an ordered group. The next result is easy to be proved.

\begin{proposition}\label{assprime}
For each $k,n\in\N$, $n\geq2$ the algebra $\K_k(A)$ is a $\zn$-prime algebra.
\end{proposition}

We consider $Q$ to be the $\zn$-graded algebra of central quotients of $\K_k(A)$, always existing due to Proposition \ref{assprime}. We denote by $Z$ the (graded)-center of $Q$ which is a (graded) field. We observe that $Z$ contains $\K$. Following word by word the computations by Procesi (see \cite{pro1}, Part II, Section 1), we have the next result. 

\begin{proposition}\label{gradedone}
Let $k,n\in\N$, where $k\geq2$, then \[\text{\rm
tr.deg$_\mathbb{K}Z=k(n^2-1)-n+1$}.\]
\end{proposition}

Because $L_k(A)$ is contained in $\K_k(A)$, due to Proposition \ref{gradedone} and Theorem 5.7 of \cite{cen2}, we have the following.

\begin{proposition}\label{graded}
Let $k,n\in\N$, where $k\geq2$, then \[\gk_k^{sl_n}\leq k(n^2-1)-n+1.\]
\end{proposition}

We observe that the growth function $g(r)$ of $L_k(A)$ depends on the number of linearly independent polynomials of degree $r$ appearing in one of its entries, to say $(a,b)$, hence from now on we are going to investigate such a number in such an entry. Let us set \[X^*=\{x_{ij}^{(r)}| i,j=1,\ldots,n-1,\ i\neq j,\ r=1,\ldots,k\}.\] We observe that for each polynomial $m$ of any degree appearing in the entry $(a,b)$ of $\K_k(A)$ having no variables appearing in one of the $A_i^0$ (so in the variables from $X^*$), we can construct a non-zero monomial $M$ of the same degree of $L_k(A)$ having $m$ as one of its summands in its $(a,b)$-entry. More precisely, if $m$ is generated by the product $A_{i_1}\cdots A_{i_r}$ we consider $M=[A_{i_1},\ldots,A_{i_r}]$. 

Let $\sigma\in S_r$, then we consider the natural left action of $GL_r$ on $A_{i_1}\cdots A_{i_r}$ such that $\sigma(A_{i_1}\cdots A_{i_r})=A_{i_{\sigma_{-1}(1)}}\cdots A_{i_{\sigma^{-1}(r)}}$. We observe that $M=\sum_{\varphi\in\Phi}\sigma(A_{i_1}\cdots A_{i_r})$, where $|\Phi|=2^{r-1}$ and each permutation of $\Phi$ is a product of cycles in which at most one cycle has length greater than 2. We have the next result.  

\begin{proposition}
In the previous notation and hipothesis, the left-normed Lie monomial $M=[A_{i_1},\ldots,A_{i_r}]$ contains $m$ in its $(a,b)$-entry.
\end{proposition}
\proof
Straightforward computations show that the result is true for $r=2$. Let $r>2$, then suppose the proposition true for $r-1$. We observe that the $(a,b)$ entry of $M$ is $p_{ab'}x_{b'b}-x_{ab''}p_{b''b}$, where $p_{ab'},p_{b''b}$ are entries of $[A_{i_1},\ldots,A_{i_{r-1}}]$. Due to the induction hipotheses, $p_{ab'}=m_{ab'}+p'$, where $m_{ab'}$ is the monomial appearing in the $(a,b')$-entry of $A_{i_1}\cdot\cdots A_{i_{r-1}}$ and $p'$ is a polynomial. If the assertion is not true, there exists a monomial $x_{b''b'''}\cdots x_{a'b}$ of $p_{b''b}$ such that $m_{ab'}x_{b'b}-x_{ab''}x_{b''b'''}\cdots x_{a'b}=0$. We note that at least the first and last variable of $m_{ab'}$ are different than $x_{b''b'''}$ and $x_{a'b}$. It turns out that there exist a submonomial $l:=x_{ij}\cdots x_{a'b}$ of $x_{ab''}x_{b''b'''}\cdots x_{a'b}$ such that $x_{ij}\cdots x_{a'b}x_{ab''}x_{b''b'''}\cdots x_{a''j}=m_{ab'}x_{b'b}$ that means $a=b=i$. It also means that there exists a permutation, which is a product of two cycles, $\varphi\in\Phi$ such that $x_{ij}\cdots x_{a'b}x_{ab''}x_{b''b'''}\cdots x_{a''j}$ and $m_{ab'}x_{b'b}$ are the $(a,a)$ entry of $\varphi(A_{i_1},\ldots,A_{i_{r-1}})$ and $A_{i_1},\ldots,A_{i_{r-1}}$ respectively, i.e., the latter are linearly dependent. Because each permutation of $\Phi$ is a product of cycles in which at most one cycle has length greater than 2, we have $\varphi(A_{i_1},\ldots,A_{i_{r-1}})$ and $A_{i_1},\ldots,A_{i_{r-1}}$ are linearly dependent if and only if $m_{ab'}=l$, that is (see \cite{vas1}) if and only if $A_{i_r}$ has $\Z_n$-degree 0.
\endproof

From now on we shall consider $sl_n$, $n\geq 3$. We observe that left-normed monomials with different multidegree are linearly independent. We consider the ordered set $S=\{m_1,\ldots,m_s\}$ of linearly independent monomials of degree $r$ appearing in the entry $(a,b)$ of a product of generators of $\K_k(A)$ of $\Z_n$-degree different than 0.  Due to the previous observation, we may consider the $m_i$'s as generated by the non-zero monomials $M_1,\ldots,M_s$ in $\K_k(A)$ with the same multidegree, i.e., if $m_1$ is generated by $M_1:=A_{i_1}\cdots A_{i_r}$, then for each $t\in\{2,\ldots,s\}$ there exists $\sigma_t\in S_r$ such that $m_t$ is generated by $\sigma_t(M_1):=M_t=A_{i_{\sigma_t^{-1}(1)}}\cdots A_{i_{\sigma_t^{-1}(r)}}$. We take now the set $S'$ of non-zero left-normed monomials $L_t$ of $L_k(A)$ as above corresponding to each $t$. We observe that each summand of $L_t$ is a monomial of $\K_k(A)$ obtained by $M_t$ applying a permutation $\varphi\in \Phi$. 

\

Due to the fact that linearly independent monomials of $\K_k(A)$ grow polynomially whereas the monomial identities of a single monomial grow factorially, for sufficiently large $r$ and for each $M_i$, $M_j$, it is possible to find permutations $\sigma,\tau$ such that $\sigma(M_i)=M_i$, $\tau(M_j)=M_j$ and $\varphi(\sigma(M_i))\neq\varphi(\tau(M_j))$, $\varphi\in\Phi$. Hence we have the following result.

\begin{proposition}\label{propositionfin}
For sufficiently large $r$, $S'$ is a linearly independent set.
\end{proposition}


In an analogous way, we consider now polynomials $m$ of any degree appearing in the entry $(a,b)$ of $\K_k(A)$ in the variables from $X$. We observe that $\gk(\K_k(A))=\gk(\mathcal{M})$, where $\mathcal{M}$ is the finitely generated $\K[X]$-module contained in $\K_k(A)$ formed by all products $A_{i_1}\cdots A_{i_r}$ such that $A_{i_1}\neq A_{i_2}$ and for each initial submonomial $A_{i_1}\cdots A_{i_l}$, $l\leq r$ such that $A_{i_1}\cdots A_{i_l}$ has $\Z_n$-degree 0 we have that $A_{i_{l+1}}$ has not $\Z_n$-degree 0. In light of this, we have that Proposition \ref{propositionfin} gives us the next.

\begin{proposition}\label{last}
Let $k\in\N$, then $\gk_k^{\Z_n}(sl_n)\geq \gk(\mathcal{M})$.
\end{proposition}

Combining Propositions \ref{graded} and \ref{last} we have the following result.

\begin{theorem}\label{final}
Let $k\in\N$, then $\gk_k^{\Z_n}(sl_n)=k(n^2-1)-n+1$.
\end{theorem}

\section{Conclusions}

We want to draw out a parallelism with the arguments used by Procesi and the first of the authors in order to compute the $\Z_n$-graded Gelfand-Kirillov dimension of $M_n(\K)$ (see \cite{pro1} and \cite{cen2}).

\

We introduce the notion of \textit{$G$-prime (Lie) algebra}.

\begin{definition} A $G$-graded ideal $P$ of $L$ is called \textit{$G$-prime}
if $aH \subseteq P$ for $G$-homogeneous element $a \in L$ and a
$G$-graded ideal $H$ of $L$, then either $a \in P$ or $H \subseteq
P$. If $(0)$ is a $G$-prime we say that $L$ is a \textit{$G$-prime algebra}.
\end{definition}

When the grading group is not specified, we shall refer to \textit{graded prime} algebras.

\

We have that the relatively-free algebras of $sl_2$ endowed with the gradings introduced in the previous section are graded prime.

\begin{proposition}\label{uno} The relatively-free $G$-graded algebra of $sl_2$ in $k$
variables $\mathcal{L}_k^G(sl_2)$ is $G$-prime.
\end{proposition}

\begin{proof} We argue only in the case $G=\Z_2$ because the other cases may be treated analogously. Suppose that $aH = 0$ for a $\mathbb{Z}_2$-homogeneous element
$a$ of degree $i$ and a $\mathbb{Z}_2$-graded ideal $H \neq 0$ of
$\mathcal{L}_k^{\mathbb{Z}_2}(sl_2)$. Since $H = H_0 \oplus H_1$ is
a non-trivial $\mathbb{Z}_2$-graded ideal then $H_0 \neq 0$ and $H_1
\neq 0$ where $H_0$ and $H_1$ are the homogeneous component of $H$
of the degree $0$ and $1$ respectively. By simple calculations $aH_j
= 0$ with $j \neq i$ implies that $a = 0$.
\end{proof}

We have already pointed out that the associative algebras $\K_k(A)$ generated by the generic homogeneous elements of the $\Z_n$-graded relatively-free algebra of $sl_n$ are $\Z_n$-prime (Proposition \ref{assprime}). Of course, due to the theory of Balaba \cite{bal1}, we have the existence of a (graded) field of quotients whose transcendence degree over the ground field equals the GK dimension of $\K_k(A)$. The fact that the relatively-free algebras $L_k(A)$ of $sl_n$ is prime in a Lie sense (Proposition \ref{uno}) gives us a hope that what happened in the associative case is maybe true in the Lie case, i.e. there exists a graded central localization of $L_k(A)$ such that the ``generalized transcendence degree'' of its ``center'' over the ground field gives a measure of the graded GK dimension of $L_k(A)$. We recall that a definition of central localization for Lie algebras may be found in \cite{mol1}.

For the sake of completeness we want to show a consequence of some results of Procesi and Razmyslov (see \cite{pro2} and \cite{razmyslov}) in the ordinary case. In particular, let $C_{kn}^{(0)}$ be the trace algebra of $\K_k(A)$, the associative unitary algebra generated by $k$ ungraded generic traceless $n\times n$ matrices over the field $\K$, and $T_{kn}^{(0)}$ the mixed trace algebra of $\K_k$. Then we have the next result.

\begin{theorem}\label{final2}
For $k,n\geq2$ we have \[\gk(T_{kn}^{(0)})=\gk(C_{kn}^{(0)})=(k^2-1)(n-1).\]
\end{theorem}

Due to the fact that $\K_k(A)$ is a prime algebra, we have the ungraded analog of Proposition \ref{gradedone}.

\begin{proposition}\label{ordinaryone}
Let $k,n\in\N$, where $k\geq2$, then \[\gk(\K_k(A))=(k^2-1)(n-1).\]
\end{proposition}

In light of Theorem \ref{final} and Proposition \ref{ordinaryone}, it is reasonable to consider a graded generalization of Theorem \ref{final2}.

\end{document}